\documentclass[12pt]{amsart}
\usepackage{amsthm,amsmath,amssymb,enumerate}
\newtheorem{thm}{Theorem}[section]
\newtheorem{prop}[thm]{Proposition}
\newtheorem{lemm}[thm]{Lemma}
\newtheorem{coro}[thm]{Corollary}
\theoremstyle{definition}
\newtheorem*{rema}{Remark}

\newcommand{\Z}{\mathbb{Z}}
\newcommand{\Aut}{\operatorname{Aut}}
\newcommand{\Hol}{\operatorname{Hol}}
\newcommand{\ci}{\circ}

\begin{document}
\title[Classification of skew left braces]
{Classification of skew left braces with additive group
  isomorphic to the infinite dihedral group}

\author[A. Hanaki]{Akihide Hanaki}
\address{Faculty of Science, Shinshu University,
  Matsumoto 3-1-1, 390-8621, Japan}
\email{hanaki@shinshu-u.ac.jp}

\author[Y. Sakata]{Yuto Sakata}
\address{1-6-35 Miyanomori, Muroran,
  Hokkaido, 050-0073, Japan}
\email{yuto.sakata.math@gmail.com}

\author[H. Yoshino]{Hiroki Yoshino}
\address{2642 Okitsu, Katsuura, Chiba 299-5245, Japan}
\email{yoshino.math@gmail.com}

\begin{abstract}
  We classify all skew left braces with additive group
  isomorphic to the infinite dihedral group.
  There are ten isomorphism classes.
\end{abstract}

\maketitle

\section{Introduction}
Let $A$ be a set with two binary operations
$(a,b)\mapsto a\cdot b$ and $(a,b)\mapsto a \ci b$.
The triple $(A,\cdot,\ci)$ is called a \emph{skew left brace} if
$(A,\cdot)$ and $(A,\ci)$ are groups and
$$a\ci(b\cdot c)=(a\ci b)\cdot a^{-1} \cdot (a\ci c)$$
holds for all $a,b,c\in A$,
where $a^{-1}$ is the inverse element of $a$
with respect to the operation of $(A,\cdot)$.
The groups $(A,\cdot)$ and $(A,\ci)$ are called the \emph{additive group}
and the \emph{multiplicative group} of the skew left brace, respectively.
We remark that the additive group is not necessarily an abelian group.
(If $(A,\cdot)$ is abelian, then $(A,\cdot,\ci)$ is called a \emph{left brace}.)
We often abbreviate $a\cdot b$ as $ab$.
The identity element of $(A,\cdot)$ is also the identity element of $(A,\ci)$.

Let $G$ be a group.
It is natural to ask how many skew left braces with additive (multiplicative) group $G$ exist.
In this direction, there are many results for finite groups $G$
\cite{MR4414157, MR4391827, MR4815260, MR4678589, MR4113853,
  MR4513787, MR4950651, MR3647970}.
For the infinite cyclic group $G$ as an additive group,
it is known that there are exactly two isomorphism
classes of skew left braces \cite[Proof of Proposition 6]{MR2298848}.
In this paper, we determine all isomorphism classes of skew left braces
with additive group isomorphic to the infinite dihedral group $D_\infty$.
There are ten isomorphism classes of such skew left braces
(Theorem \ref{cls_slb}).
This answers a problem raised by Vendramin
\cite[after Problem 2.27]{MR3974481}.

It is known that skew left braces with additive group $G$ correspond to
regular subgroups of the holomorph of $G$, and they correspond to
maps $\gamma:G\to \Aut(G)$ with some property
\cite[Theorem 2.2]{MR4089566}, \cite[Theorem 4.2]{MR3647970}.
We will consider these maps to classify skew left braces.
Our method is very simple,
but it requires a lot of case analysis and calculations.

\section{Preliminaries}
\subsection{Holomorphs and $\gamma$-maps}
Let $G$ be a group, and consider the automorphism group $\Aut(G)$.
The semidirect product of $G$ and $\Aut(G)$ is called
the \emph{holomorph} of $G$, and denoted by $\Hol(G)$.
Namely,
$$\Hol(G)=G\rtimes \Aut(G)=\{(a,f)\mid a\in G,\ f\in \Aut(G)\},$$
and the product is defined by
$$(a,f)(b,g)=(af(b), fg).$$
The holomorph acts on $G$ from the left by
$(a,f)x=af(x)$ for $(a,f)\in \Hol(G)$ and $x\in G$.

Let $G$ be a group.
A map $\gamma:G\to\Aut(G)$ is called a \emph{$\gamma$-map} on $G$ if
$$\gamma(a\gamma(a)(b))=\gamma(a)\gamma(b)$$
holds for all $a, b\in G$.
The automorphism group $\Aut(G)$ acts on the set of all $\gamma$-maps by
$$\gamma^\varphi(a)=\varphi^{-1}\gamma(\varphi(a))\varphi\quad
(a\in G,\ \varphi\in\Aut(G)).$$
For two $\gamma$-maps $\gamma$ and $\gamma'$,
we define $\gamma\sim\gamma'$
if there exists $\varphi\in \Aut(G)$ such that $\gamma'=\gamma^\varphi$.
The relation $\sim$ is an equivalence relation
on the set of all $\gamma$-maps on $G$.
  
Now we are ready to give a correspondence between skew left braces
with additive group $G$ and some other objects
\cite[Theorem 2.2]{MR4089566}, \cite[Theorem 4.2]{MR3647970}.

\begin{prop}\label{prop2.1}
  Let $G$ be a group.
  There are bijections between the following sets:
  \begin{enumerate}[{\rm (1)}]
    \item The set of isomorphism classes of skew left braces
    with additive group $G$.
    \item The set of conjugacy classes of regular subgroups of $\Hol(G)$.
    \item The set of $\sim$ equivalence classes of $\gamma$-maps on $G$. 
  \end{enumerate}
\end{prop}

For any $\gamma$-map $\gamma:G\to \Aut(G,\cdot)$,
the binary operation
$$a\circ b = a \cdot \gamma(a)(b) \qquad (a,b\in G)$$
makes $(G,\cdot,\circ)$ into a skew left brace.

\subsection{The infinite dihedral group $D_\infty$}
The infinite dihedral group is
$$D_\infty=\langle x,y\mid y^2=1,\ yx=x^{-1}y\rangle.$$
We denote $x^ay^s$ by $(a,s)$, and so
$$D_\infty=\{(a,s)\mid a\in \Z,\ s\in \Z_2\},$$
where $\Z_2=\Z/2\Z$, with operation
$$(a, s)(b, t)=(a+(-1)^s b,\ s+t).$$

For the automorphism group of $D_\infty$,
we have the following proposition.

\begin{prop}\label{prop2.2}
  We have
  $\Aut(D_\infty)=\{[b,t]\mid b\in \Z,\ t\in\Z_2\}$,
  where $$[b,t](a,s)=((-1)^ta+bs,s).$$
  The following equations hold.
  \begin{enumerate}[{\rm (1)}]
    \item $[a,s][b,t] = [a+(-1)^s b, s+t]$,
    \item $[a,s]^{-1}=[(-1)^{s+1}a,s]$,
    \item $[b,t]^{-1}[a,s][b,t]=[(-1)^ta+(-1)^{t+1}b+(-1)^{s+t}b,s]$.
  \end{enumerate}
  (By (1), we can see $\Aut(D_\infty)\cong D_\infty$ by $[b,t]\mapsto (b,t)$.)
\end{prop}

\begin{proof}
  The order of $(a,s)\in D_\infty$ is $\infty$ if and only if $a\ne 0$ and $s=0$,
  and the order is $2$ if $s=1$.
  
  Let $f$ be an automorphism of $D_\infty$.
  Since only $(1,0)$ or $(-1,0)$ generates $\{(a,0)\mid a\in \Z\}$, $f(1,0)$ must be $(1,0)$ or $(-1,0)$.
  There is a pair $(b,t)\in \Z\times \Z_2$ such that $f(1,0)=((-1)^t,0)$ and $f(0,1)=(b,1)$.
  
  Conversely, for $(b,t)\in \Z\times \Z_2$,
  we define $f$ by  $f(1,0)=((-1)^t,0)$ and $f(0,1)=(b,1)$.
  Easily we can check that $f$ gives an automorphism of $D_\infty$.

  We can conclude that $\Aut(D_\infty)=\{[b,t]\mid b\in \Z,\ t\in\Z_2\}$.
  The remaining equations are clear.
\end{proof}

We remark that $s\in \Z_2$ can appear in
the first entry of $[b,t](a,s)=((-1)^ta+bs,s)$.
When an element of $\Z_2$ appears in the first entry of a pair, the classes of $0$
and $1$ should be intended as the corresponding integers.
However, when $(s+t)$ appears for $s=t=1$ in the first entry,
it should be intended as the integer $0$.
We will use the similar notation throughout this paper.

\section{Main results -- skew left braces with additive group $D_\infty$}
We classify all skew left braces with additive group $D_\infty$.
By Proposition \ref{prop2.1},
we classify equivalence classes of $\gamma$-maps on $D_\infty$.
We determine all $\gamma$-maps, and then we consider their equivalence.
Consequently we will get ten isomorphism classes of such skew left braces.
The following theorem is the main result of this paper.

\begin{thm}\label{cls_slb}
  Let $(A,\cdot,\ci)$ be a skew left brace with $(A,\cdot)=D_\infty$.
  Then $(A,\cdot,\ci)$ is isomorphic to a skew left brace
  defined by one of the following multiplications.
  Conversely, the following ten operations define
  non-isomorphic skew left braces.
  \begin{itemize}
    \item[(1)] $(a,s)\circ(b,t)=(a,s)(b,t)=(a+(-1)^sb,s+t)$
    (the trivial skew brace)
    \item[(2,0)] $(a,s)\circ(b,t)=(a+(-1)^sb-2at,s+t)$
    \item[(2,1)] $(a,s)\circ(b,t)=(a+(-1)^sb-2at+(-1)^s st,s+t)$
    \item[(3)] $(a,s)\circ(b,t)=(a+b-2at,s+t)$
    \item[(4,0)] $(a,s)\circ(b,t)=(a+b,s+t)$
    \item[(4,1)] $(a,s)\circ(b,t)=(a+b+(-1)^s st,s+t)$
    \item[(5,0)] $(a,s)\circ(b,t)=\begin{cases}
      (a+(-1)^sb,s+t), & \text{if $a$ is even},\\
      (a+(-1)^{s+1}b,s+t), & \text{if $a$ is odd}.
    \end{cases}$
    \item[(5,1)] $(a,s)\circ(b,t)=\begin{cases}
      (a+(-1)^sb,s+t), & \text{if $a$ is even},\\
      (a+(-1)^{s+1}b+(-1)^s2t,s+t), & \text{if $a$ is odd}.
    \end{cases}$
    \item[(6,0)] $(a,s)\circ(b,t)=\begin{cases}
      (a+b-2at,s+t), & \text{if $a+s$ is even},\\
      (a-b-2at+2t,s+t), & \text{if $a+s$ is odd}.
    \end{cases}$
    \item[(6,1)]
    $(a,s)\circ(b,t)=\begin{cases}(a+b-2at,s+t), &\text{if $a+s$ is even,}\\
      (a-b-2at+4t,s+t), &\text{if $a+s$ is odd.}\end{cases}$
  \end{itemize}
\end{thm}

Theorem \ref{cls_slb} will be proven by applying $a\ci b=a\gamma(a)(b)$ to
the classification of $\gamma$-maps (Theorem \ref{cls_gamma}).
For the structure of the multiplicative groups $(A,\ci)$,
we have the following corollary.

\begin{coro}
  The multiplicative groups $(A,\ci)$ of the skew left braces
  $(A,\cdot,\ci)$ in Theorem \ref{cls_slb} are isomorphic to 
  \begin{itemize}
    \item $\Z$ for (2,1), (4,1),
    \item $\Z\times \Z_2$ for (2,0), (4,0), 
    \item $D_\infty=\Z\rtimes \Z_2$ for (1), (3), (5,0), (6,0),
    \item $D_\infty\times \Z_2=(\Z\rtimes \Z_2)\times \Z_2$ for (5,1), (6,1).
  \end{itemize}
\end{coro}

\begin{proof}
  We write by $X^{\circ a}$ the $a$-th power of $X$ with respect to the operator $\circ$.

  Write $$\Z = \langle x\rangle.$$

  \begin{itemize}
    \item[(2,1)] Set $X=(0,1)$. We have $X^{\circ (-1)}=(-1,1)$.
    By induction to positive and negative directions, we can show that
    $$X^{\circ a}=\begin{cases}\left(\frac{a-1}{2},1\right), &\text{if $a$ is odd},\\
      \left(-\frac{a}{2},0\right), &\text{if $a$ is even}.\end{cases}$$
    This shows that the map defined by $x\mapsto X$
    is an isomorphism from $\Z$ to the group.

    \item[(4,1)] Set $X=(0,1)$. We have $X^{\circ (-1)}=(1,1)$.
    By induction, we can show that
    $$X^{\circ a}=\begin{cases}\left(\frac{-a+1}{2},1\right), &\text{if $a$ is odd},\\
      \left(-\frac{a}{2},0\right), &\text{if $a$ is even}.\end{cases}$$
    This shows that the map defined by $x\mapsto X$
    is an isomorphism from $\Z$ to the group. 
  \end{itemize}

  Write $$\Z\times \Z_2 = \langle x,y\mid y^2=[x,y]=1\rangle.$$
  \begin{itemize}
    \item[(2,0)] Set $X=(1,0)$ and $Y=(0,1)$.
    By induction, $X^{\circ a}=(a,0)$ for $a\in \Z$.
    Directly we can show that
    $$X^{\circ a}\circ Y=(-a,1)=Y\circ X^{\circ a}.$$
    These show that the map defined by $x\mapsto X$ and $y\mapsto Y$
    is an isomorphism from $\Z\times \Z_2$ to the group.
    
    \item[(4,0)] This is an easy case.
    The map defined by $x\mapsto X=(1,0)$, $y\mapsto Y=(0,1)$
    is an isomorphism from $\Z\times \Z_2$ to the group.
  \end{itemize}

  Write $$D_\infty=\Z\rtimes \Z_2 = \langle x,y\mid y^2=1, yx=x^{-1}y\rangle.$$
  \begin{itemize}
    \item[(1)] This is the trivial case.
    The map defined by $x\mapsto X=(1,0)$, $y\mapsto Y=(0,1)$
    is an isomorphism from $D_\infty$ to the group.

    \item[(3)] Set $X=(1,0)$ and $Y=(0,1)$.
    We can see $X^{\circ a}=(a,0)$, $Y^{\circ 2}=(0,0)$, $Y\circ X=(1,1)=X^{\circ (-1)}\circ Y$, and
    $X^{\circ a}\circ Y=(-a,1)$ ($a\in \Z$).
    Thus the map defined by $x\mapsto X$ and $y\mapsto Y$ is an isomorphism from
    $D_\infty$ to the group.

    \item[(5,0)] Set $X=(1,1)$ and $Y=(0,1)$. We have $X^{\circ (-1)}=(-1,1)$.
    We can see $X^{\circ a}=(a,a)$, $Y^{\circ 2}=(0,0)$, $Y\circ X=(-1,0)=X^{\circ (-1)}\circ Y$, and
    $X^{\circ a}\circ Y=(a,a+1)$ ($a\in \Z$).
    Thus the map defined by $x\mapsto X$ and $y\mapsto Y$ is an isomorphism from
    $D_\infty$ to the group.
    
    \item[(6,0)] Set $X=(0,1)$ and $Y=(1,0)$. We have $X^{\circ (-1)}=(2,1)$.
    We can check the relations $Y^{\circ 2}=(0,0)$, $Y\circ X=X^{\circ (-1)}\circ Y$.
    By induction, for $a\in \Z$, we can show that
    \begin{align*}
      X^{\circ a} &=
            \begin{cases}
              (-a+1,1), & \text{if $a$ is odd},\\
              (a,0), & \text{if $a$ is even},\\
            \end{cases}\\
      X^{\circ a}\circ Y &=
             \begin{cases}
               (-a,1), & \text{if $a$ is odd},\\
               (a+1,0), & \text{if $a$ is even}.
             \end{cases}
    \end{align*}
    We have an isomorphism from $D_\infty$ to the group.
  \end{itemize}

  Write
  \begin{align*}D_\infty\times \Z_2&=(\Z\rtimes \Z_2)\times \Z_2\\
                                   &= \langle x,y,z\mid y^2=z^2=1, yx=x^{-1}y, [x,z]=[y,z]=1\rangle.
  \end{align*}
  \begin{itemize}
    \item[(5,1)]  Set $X=(2,0)$, $Y=(1,0)$, $Z=(1,1)$.
    We have $X^{\circ (-1)}=(-2,0)$.
    We can check the relations $Y^{\circ 2}=Z^{\circ 2}=(0,0)$,
    $X\circ Z=Z\circ X$, $Y\circ Z=Z\circ Y$, and 
    $Y\circ X=X^{\circ (-1)}\circ Y$.
    For $a\in \Z$, we can show the following equations by induction:
    \begin{align*}
      X^{\circ a} &= (2a,0),\\
      X^{\circ a}\circ Y &= (2a+1,0),\\
      X^{\circ a}\circ Z &= (2a+1,1),\\
      X^{\circ a}\circ Y\circ Z &= (2a+2,1).
    \end{align*}
    We have an isomorphism from $D_\infty\times \Z_2$ to the group.
    
    \item[(6,1)] Set $X=(0,1)$, $Y=(1,0)$, $Z=(2,1)$.
    We have $X^{\circ (-1)}=(-2,0)$.
    We can check the relations $Y^{\circ 2}=Z^{\circ 2}=(0,0)$,
    $X\circ Z=Z\circ X$, $Y\circ Z=Z\circ Y$, and 
    $Y\circ X=X^{\circ (-1)}\circ Y$.
    For $a\in \Z$, we can show the following equations by induction:
    \begin{align*}
      X^{\circ a} &=
                    \begin{cases}
                      (-2a+2,1), & \text{if $a$ is odd},\\
                      (2a,0), & \text{if $a$ is even},
                    \end{cases}\\
      X^{\circ a}\circ Y &=
                           \begin{cases}
                             (-2a+1,1), & \text{if $a$ is odd},\\
                             (2a+1,0), & \text{if $a$ is even},
                           \end{cases}\\
      X^{\circ a}\circ Z &=
                           \begin{cases}
                             (2a,0), & \text{if $a$ is odd},\\
                             (-2a+2,1), & \text{if $a$ is even},
                           \end{cases}\\
      X^{\circ a}\circ Y\circ Z &=
                                  \begin{cases}
                                    (2a+1,0), & \text{if $a$ is odd},\\
                                    (-2a+1,1), & \text{if $a$ is even}.
                                  \end{cases}
    \end{align*}
    The map defined by $x\mapsto X$, $y\mapsto Y$, and $z\mapsto Z$ is
    an isomorphism from $D_\infty\times \Z_2$ to the group.
  \end{itemize}
\end{proof}

\section{$\gamma$-Maps}
In this section, we will determine all $\gamma$-maps on $D_\infty$.
We will consider their equivalence in the next section.
For a map $\gamma:D_\infty\to\Aut(D_\infty)$, we write
$$\gamma(a,s)=[p(a,s),q(a,s)],$$
where $p:D_\infty\to \Z$ and $q:D_\infty\to \Z_2$.
Easily, we can see that $p(0,0)=0$ and $q(0,0)=0$,
if $\gamma$ is a $\gamma$-map.

\begin{prop}\label{prop_gamma}
  A map $\gamma:D_\infty\to \Aut(D_\infty)$ is a $\gamma$-map if and only if
  the following two equations hold for all $(a,s), (b,t)\in D_\infty$:
  \begin{align}
    p(a,s)+(-1)^{q(a,s)}p(b,t)
    &=p(a+(-1)^{s+q(a,s)}b+(-1)^sp(a,s)t,s+t),\tag{$\star$}\\
    q(a,s)+q(b,t)
    &=q(a+(-1)^{s+q(a,s)}b+(-1)^sp(a,s)t,s+t).\tag{$\star\star$}
  \end{align}
  We remark that the second equation is in $\Z_2$.
\end{prop}

\begin{proof}
  The result holds by the definition of $\gamma$-maps.
\end{proof}

We divide our arguments into two cases:
the case $q(1,0)=0$ and the case $q(1,0)=1$.

\subsection{The case $q(1,0)=0$}
In this subsection, we assume $q(1,0)=0$.

\begin{lemm}\label{lem4.2}
  If $\gamma$ is a $\gamma$-map,
  then $p(b,0)=bp(1,0)$, $q(b,0)=0$,
  $p(b,1)=p(0,1)-bp(1,0)$, and $q(b,1)=q(0,1)$
  for all $b\in \Z$.
\end{lemm}

\begin{proof}
  Set $(a,s)=(1,0)$ and $t=0$ for Proposition \ref{prop_gamma}.
  By $q(1,0)=0$, we have
  $p(1,0)+p(b,0)=p(1+b,0)$ and $q(1,0)+q(b,0)=q(1+b,0)$.
  Thus $p(b,0)=bp(1,0)$ and $q(b,0)=0$.

  Set $(a,s)=(0,1)$ and $t=0$ for Proposition \ref{prop_gamma} ($\star$).
  We have
  $p(0,1)+(-1)^{q(0,1)}bp(1,0)=p((-1)^{1+q(0,1)}b,1)$.
  Thus $p(b,1)=p(0,1)-bp(1,0)$.

  Set $(a,s)=(0,1)$ and $t=0$ for Proposition \ref{prop_gamma} ($\star\star$).
  We have
  $q(0,1)+q(b,0)=q((-1)^{1+q(0,1)}b,1)$.
  Thus $q(b,1)=q(0,1)+q((-1)^{1+q(0,1)}b,0)=q(0,1)$.
\end{proof}

\begin{lemm}\label{lem4.3}
  If $\gamma$ is a $\gamma$-map,
  then 
  $p(1,0)=0$ or $-2$. Moreover, in this case,
  \begin{enumerate}[{\rm (1)}]
    \item if $q(0,1)=0$,
    then $q(a,s)=0$ and $(p(1,0),p(0,1))=(0,0)$ or $(-2, k)$ for some $k\in \Z$,
    \item if $q(0,1)=1$,
    then $q(a,s)=s$ and $(p(1,0),p(0,1))=(-2,0)$ or $(0, k)$ for some $k\in \Z$.
  \end{enumerate}
  For each case, the condition ($\star\star$) in Proposition \ref{prop_gamma} holds.
\end{lemm}

\begin{proof}
  By Lemma \ref{lem4.2}, $q(a,s)$ is determined,
  and clearly ($\star\star$) in Proposition \ref{prop_gamma} holds.
  
  Set $(a,s)=(1,0)$ and $(b,t)=(0,1)$ for ($\star$).
  We have $p(1,0)+p(0,1)=p(1+0+p(1,0),1)$ and so
  $p(1,0)(2+p(1,0))=0$. This shows $p(1,0)=0$ or $-2$.
  
  Set $(a,s)=(b,t)=(0,1)$ for ($\star$).
  We have $p(0,1)+(-1)^{q(0,1)}p(0,1)=p(-p(0,1),0)=-p(0,1)p(1,0)$.
  If $q(0,1)=0$, then $p(0,1)(2+p(1,0))=0$
  and this means $p(0,1)=0$ or $p(1,0)=-2$.
  If $q(0,1)=1$, then $p(0,1)p(1,0)=0$
  and this means $p(0,1)=0$ or $p(1,0)=0$.
\end{proof}

\begin{prop}\label{prop00}
  Suppose $\gamma$ is a $\gamma$-map and $q(1,0)=q(0,1)=0$.
  Then one of the following statements holds:
  \begin{enumerate}[{\rm (1)}]
    \item $p(a,s)=0$, $q(a,s)=0$,
    \item there exists $k\in\Z$ such that
    $p(a,s)=sk+(-1)^{s+1}2a$ (namely $p(a,0)=-2a$ and $p(a,1)=k+2a$), $q(a,s)=0$.
  \end{enumerate}
  Conversely, if one of the above holds, then $\gamma$ is a $\gamma$-map.
\end{prop}

\begin{proof}
  Suppose $\gamma$ is a $\gamma$-map.
  Consider the case in Lemma \ref{lem4.3} (1).
  By Lemma \ref{lem4.3}, $q(a,s)=0$.
  If $(p(1,0),p(0,1))=(0,0)$, then (1) holds.
  If $(p(1,0),p(0,1))=(-2,k)$, then
  we can write $p(a,s)=sk+(-1)^{s+1}2a$ and (2) holds.

  Suppose that $\gamma$ is defined by (1) or (2).
  For (1),  the conditions in Proposition \ref{prop_gamma} are clear.
  We consider ($\star$) for (2).
  \begin{align*}
    \mathrm{LHS} &= sk+(-1)^{s+1}2a+tk+(-1)^{t+1}2b,\\
    \mathrm{RHS} &= p(a+(-1)^sb+(-1)^sp(a,s)t,s+t)\\
                 &= p(a+(-1)^sb+(-1)^s(sk+(-1)^{s+1}2a)t,s+t)\\
                 &= (s+t)k+(-1)^{s+t+1}2(a+(-1)^sb+(-1)^s(sk+(-1)^{s+1}2a)t)\\
                 &= (s+t)k+(-1)^{s+t+1}2a+(-1)^{t+1} 2b
                   +(-1)^{t+1}2stk+(-1)^{s+t}4at.
  \end{align*}
  We remark that $(s+t)=0$ if $s=t=1$.
  
  If $t=0$, then
  $$\mathrm{LHS} = sk+(-1)^{s+1}2a-2b= \mathrm{RHS}.$$ 
  If $(s,t)=(0,1)$, then
  \begin{align*}
    \mathrm{LHS} &= -2a+k+2b,\\
    \mathrm{RHS} &= k+2a+2b-4a = -2a+k+2b.
  \end{align*}
  If $(s,t)=(1,1)$, then
  \begin{align*}
    \mathrm{LHS} &= 2k+2a+2b,\\
    \mathrm{RHS} &= -2a+2b+2k+4a=2k+2a+2b.
  \end{align*}
  For every case, the equation ($\star$) holds.
  The equation ($\star\star$) for (2) holds by Lemma \ref{lem4.3}.
\end{proof}

\begin{prop}\label{prop01}
  Suppose $\gamma$ is a $\gamma$-map and $q(1,0)=0$, $q(0,1)=1$.
  Then one of the following statements holds:
  \begin{enumerate}[{\rm (1)}]
    \item $p(a,s)=(-1)^{s+1}2a$, $q(a,s)=s$,
    \item there exists $k\in\Z$ such that $p(a,s)=sk$, $q(a,s)=s$.
  \end{enumerate}
  Conversely, if one of the above holds, then $\gamma$ is a $\gamma$-map.
\end{prop}
  
\begin{proof}
  Suppose $\gamma$ is a $\gamma$-map.
  Consider the case in Lemmas \ref{lem4.2} and \ref{lem4.3} (2).
  We obtain (1) and (2) by setting $k=p(0,1)$.

  Suppose that $\gamma$ is defined by (1) or (2).
  In either case, ($\star\star$) is clear.
  Consider the condition ($\star$) for (1).
  \begin{align*}
    \mathrm{LHS} &= (-1)^{s+1}2a+(-1)^s(-1)^{t+1}2b
                 = (-1)^{s+1}2a+(-1)^{s+t+1}2b,\\
    \mathrm{RHS} &= p(a+b-2at,s+t)
                 = (-1)^{s+t+1}2(a+b-2at)\\
                 &= (-1)^{s+t+1}2a+(-1)^{s+t+1}2b+(-1)^{s+t}4at\\
                 &= (-1)^{s+t}2a(-1+2t)+(-1)^{s+t+1}2b.
  \end{align*}
  Thus ($\star$) holds. 

  For the equation ($\star$) for (2),
  \begin{align*}
    \mathrm{LHS} &= sk+(-1)^stk,\\
    \mathrm{RHS} &= p(a+b+(-1)^skst,s+t)
                     = (s+t)k.
  \end{align*}
  We remark that $s+t=0$ if $s=t=1$.
  It is easy to verify that $\mathrm{LHS}= \mathrm{RHS}$.
\end{proof}

\subsection{The case $q(1,0)=1$}
In this subsection, we assume $q(1,0)=1$.
In this subsection,  $\ell$ denotes an integer.

\begin{lemm}\label{lem4.6}
  If $\gamma$ is a $\gamma$-map,
  then 
  $p(2\ell,0)=\ell p(2,0)$,
  $p(2\ell +1, 0)=\ell p(2,0)+p(1,0)$, and
  \begin{enumerate}[{\rm (1)}]
    \item if $q(0,1)=0$, then $q(a,s)=a$,
    \item if $q(0,1)=1$, then $q(a,s)=a+s$.
  \end{enumerate}
\end{lemm}

\begin{proof}
  Set $(a,s)=(-1,0)$ and $(b,t)=(1,0)$ for  ($\star\star$).
  We get $q(-1,0)+q(1,0)=q(-1+(-1)^{q(-1,0)},0)$.
  If $q(-1,0)=0$, then $0+1=q(-1+1,0)=q(0,0)=0$ and this is a contradiction.
  Thus $q(-1,0)=1$.
  
  Set $s=0$ and $(b,t)=((-1)^{q(a,0)},0)$  for ($\star\star$). 
  We have $q(a,0)+q((-1)^{q(a,0)},0)=q(a+1,0)$.
  Since $q(a,0)\in \{0, 1\}$, $q(0,0)=0$, and $q(\pm 1, 0)=1$,
  we can conclude that $q(a,0)=a$ (in $\Z_2$).

  Set $(a,s)=(2,0)$ and $t=0$ for ($\star$).
  We get $p(2,0)+p(b,0)=p(2+b,0)$.
  This shows $p(2\ell,0)=\ell p(2,0)$ and
  $p(2\ell +1, 0)=\ell p(2,0)+p(1,0)$.

  Set $(a,s)=(0,1)$ and $t=0$ for ($\star\star$). 
  We have $q(0,1)+q(b,0)=q((-1)^{1+q(0,1)}b,1)$.
  If $q(0,1)=0$, then $b=q(b,0)=q(-b,1)$ and this shows $q(a,s)=a$.
  If $q(0,1)=1$, then $1+b=1+q(b,0)=q(b,1)$ and this shows $q(a,s)=a+s$.
\end{proof}

\begin{lemm}\label{lem4.7}
  If $\gamma$ is a $\gamma$-map,
  then 
  $p(a,s)$ is even for any $(a,s)\in D_\infty$.
\end{lemm}

\begin{proof}
  Set $(b,t)=(0,1)$ for ($\star\star$).
  We have $q(a,s)+q(0,1)=q(a+(-1)^sp(a,s),1+s)$ in $\Z_2$.
  If $q(0,1)=0$, then $a=q(a,s)=q(a+(-1)^sp(a,s),1+s)=a+(-1)^sp(a,s)$
  and this shows $p(a,s)$ is even.
  If $q(0,1)=1$, then $a+s+1=q(a,s)+1=q(a+(-1)^sp(a,s),1+s)=a+s+1+(-1)^sp(a,s)$
  and this shows $p(a,s)$ is even.
\end{proof}

We consider the case $q(0,1)=0$.

\begin{lemm}\label{lem4.8}
  If $\gamma$ is a $\gamma$-map and 
  $q(0,1)=0$, then $p(2\ell,1) = p(0,1)-\ell p(2,0)$,
  $p(2\ell+1,1)= p(0,1)-(\ell+1) p(2,0)+p(1,0)$,
  and $p(2,0)=0$ or $-4$.
\end{lemm}

\begin{proof}
  Set $(a,s)=(0,1)$ and $t=0$ for ($\star$).
  We have $p(0,1)+p(b,0)=p(-b,1)$.
  With Lemma \ref{lem4.6}, we have $p(2\ell,1) = p(0,1)-\ell p(2,0)$, and
  $p(2\ell+1,1)= p(0,1)-(\ell+1) p(2,0)+p(1,0)$.

  Set $(a,s)=(2,0)$ and $(b,t)=(0,1)$ for ($\star$).
  Since $p(2,0)$ is even by Lemma \ref{lem4.7},
  we have $p(2,0)+p(0,1)=p(2+p(2,0),1)=p(0,1)-((2+p(2,0))/2)p(2,0)$ and so
  $p(2,0)(p(2,0)+4)=0$.
  We can conclude that $p(2,0)=0$ or $-4$.
\end{proof}

\begin{prop}\label{prop10}
  If $\gamma$ is a $\gamma$-map,
  $q(1,0)=1$, and $q(0,1)=0$, then one of the following statements holds:
  \begin{enumerate}[{\rm (1)}]
    \item 
    $p(2\ell,s)= 0$,
    $p(2\ell+1,s) = 2k$, and
    $q(a,s)=a$ 
    for some $k\in \Z$,
    \item 
    $p(2\ell,0)= -4\ell$,
    $p(2\ell+1,0) = -4\ell+2k$,
    $p(2\ell,1) = 4\ell-2k-2$,
    $p(2\ell+1,1) =4\ell+2$, and
    $q(a,s)=a$
    for some $k\in \Z$.
  \end{enumerate}
  Conversely, if one of the above holds, then $\gamma$ is a $\gamma$-map.
\end{prop}

\begin{proof}
  Suppose $\gamma$ is a $\gamma$-map and $p(2,0)=0$.
  Set $(a,s)=(0,1)$ and $(b,t)=(0,1)$ for ($\star$).
  By Lemmas \ref{lem4.6}, \ref{lem4.7}, and \ref{lem4.8},
  $p(0,1)+p(0,1)=p(-p(0,1),0)=(-p(0,1)/2)p(2,0)=0$.
  Thus $p(0,1)=0$ and $q(a,s)=a$ hold for all
  $(a,s)\in D_\infty$.
  Since $p(1,0)$ is even, we can set $2k=p(1,0)$ for some $k\in\Z$. 
  We can determine all $p(a,s)$ as in (1)
  by Lemmas \ref{lem4.6} and \ref{lem4.8}.

  Suppose $\gamma$ is a $\gamma$-map and $p(2,0)=-4$.
  Set $(a,s)=(1,0)$ and $(b,t)=(0,1)$ for ($\star$).
  Since $p(1,0)$ is even, using Lemma \ref{lem4.8}, we have
  \begin{align*}
    \text{LHS} &= p(1,0)-p(0,1),\\
    \text{RHS} &= p(1+p(1,0),1) =p(0,1)-\frac{2+p(1,0)}{2} p(2,0)+p(1,0)\\
    &=p(0,1)+2(2+p(1,0))+p(1,0)=p(0,1)+3p(1,0)+4.
  \end{align*}
  We have $p(1,0)-p(0,1)=p(0,1)+3p(1,0)+4$,
  and this means $p(0,1)=-p(1,0)-2$.
  Since $p(1,0)$ is even, we can set $2k=p(1,0)$ for some $k\in\Z$.
  The equations in (2) hold.

  Suppose that $\gamma$ is defined by (1). 
  We will show the equation ($\star$) in Proposition  \ref{prop_gamma}
  for all $(a,s),(b,t)\in D_\infty$.
  If $a$ is even, then
  $\text{LHS} =p(b,t)=\text{RHS}$.
  If $b$ is even, then
  $\text{LHS} =p(a,s)=\text{RHS}$.
  If both $a$ and $b$ are odd, then
  $\text{LHS} =0=\text{RHS}$.
  The condition ($\star$) holds.
  The condition ($\star\star$) holds by Lemma \ref{lem4.6}.

  Suppose that $\gamma$ is defined by (2). 
  We can check the condition ($\star$) for all $(a,s),(b,t)\in D_\infty$.
  See Appendix A.
  The condition ($\star\star$) holds by Lemma \ref{lem4.6}.
\end{proof}

We consider the case $q(0,1)=1$.

\begin{lemm}\label{lem4.10}
  If $\gamma$ is a $\gamma$-map and
  $q(0,1)=1$, then
  $p(2\ell,1)= p(0,1)-\ell p(2,0)$,
  $p(2\ell+1,1)= p(0,1)-\ell p(2,0)-p(1,0)$,
  and $p(2,0)=0$ or $-4$.
\end{lemm}

\begin{proof}
  Set $(a,s)=(0,1)$ and $t=0$ for ($\star$).
  We have $p(0,1)-p(b,0)=p(b,1)$.
  With Lemma \ref{lem4.6}, we can conclude
  $p(2\ell,1)= p(0,1)-\ell p(2,0)$ and
  $p(2\ell+1,1)= p(0,1)-\ell p(2,0)-p(1,0)$.
  
  Set $(a,s)=(2,0)$ and $(b,t)=(0,1)$ for ($\star$).
  By Lemma \ref{lem4.7},
  we have $p(2,0)+p(0,1)=p(2+p(2,0),1)=p(0,1)-(2+p(2,0))p(2,0)/2$,
  and this shows $p(2,0)=0$ or $-4$.
\end{proof}

\begin{prop}\label{prop11}
  Suppose $\gamma$ is a $\gamma$-map.
  If $q(1,0)=q(0,1)=1$, then one of the following statements holds:
  \begin{enumerate}[{\rm (1)}]
    \item  $p(2\ell,0) = 0$,
    $p(2\ell+1,0) = 2k$,
    $p(2\ell,1) = 2k$,
    $p(2\ell+1,1) = 0$, and
    $q(a,s)= a+s$ for some $k\in \Z$,
    \item
    $p(2\ell,0) = -4\ell$,
    $p(2\ell+1,0) = -4\ell+2k$,
    $p(2\ell,1) = 4\ell$,
    $p(2\ell+1,1) = 4\ell-2k$, and
    $q(a,s)= a+s$ for some $k\in \Z$.
  \end{enumerate}
  Conversely, if one of the above holds, then $\gamma$ is a $\gamma$-map.
\end{prop}

\begin{proof}
  Suppose $\gamma$ is a $\gamma$-map
  and $p(2,0)=0$.
  Set $(a,s)=(1,0)$ and $(b,t)=(0,1)$ for ($\star$).
  We have $p(1,0)-p(0,1)=p(1+p(1,0),1)=p(0,1)-p(1,0)$.
  This shows that $p(1,0)=p(0,1)$.
  Since $p(1,0)$ is an even number, we set $p(1,0)=2k$.
  The statement (1) holds by Lemmas \ref{lem4.6} and \ref{lem4.10}.

  Suppose $\gamma$ is a $\gamma$-map
  and $p(2,0)=-4$.
  Set $(a,s)=(1,0)$ and $(b,t)=(1,1)$ for ($\star$).
  By Lemma \ref{lem4.10},
  \begin{align*}
    \text{LHS} &= p(1,0)-p(1,1)=p(1,0)-(p(0,1)-p(1,0))=2p(1,0)-p(0,1),\\
    \text{RHS} &= p(p(1,0),1)=p(0,1)-\frac{p(1,0)}{2}p(2,0)=p(0,1)+2p(1,0).
  \end{align*}
  We have $2p(1,0)-p(0,1)=p(0,1)+2p(1,0)$ and this shows $p(0,1)=0$.
  Since $p(1,0)$ is even, we set $p(1,0)=2k$.
  The statement (2) holds by Lemmas \ref{lem4.6} and \ref{lem4.10}.

  Suppose that $\gamma$ is defined by (1). 
  Note that $p(a,s)=0$ if $a+s$ is even and $p(a,s)=2k$ if $a+s$ is odd.
  The left-hand side of ($\star$) is $0$ if $(a+s)+(b+t)$ is even and
  $2k$ if $(a+s)+(b+t)$ is odd.
  Similarly, the same holds for the right-hand side.
  Similarly, the condition ($\star\star$) holds. 
  
  Suppose that $\gamma$ is defined by (2). 
  For ($\star\star$),
  \begin{align*}
    \text{LHS} &= q(a,s)+q(b,t)=a+s+b+t,\\
    \text{RHS} &= q(a+(-1)^ab+(-1)^sp(a,s)t,s+t)=a+b+s+t,
  \end{align*}
  in $\Z_2$, and the equation holds.
  For the condition ($\star$), see Appendix B.
\end{proof}

\subsection{All $\gamma$-maps}
Combining Propositions \ref{prop00}, \ref{prop01},
\ref{prop10} and \ref{prop11},
we list all $\gamma$-maps on $D_\infty$.

\begin{thm}
  $\gamma$-maps on $D_\infty$ are
  (in the following, $\ell$ is always an integer) :
  \begin{enumerate}[{\rm (1)}]
    \item
    $\gamma(a,s)=[0,0]$,
    \item
    $\gamma(a,s)=[sk+(-1)^{s+1}2a,0]$ for some $k\in \Z$,
    \item
    $\gamma(a,s)=[(-1)^{s+1}2a,s]$,
    \item
    $\gamma(a,s)=[sk,s]$ for some $k\in \Z$,
    \item 
    $\gamma(2\ell,s)=[0,0]$,
    $\gamma(2\ell+1,s)=[2k,1]$
    for some $k\in \Z$,
    \item
    $\gamma(2\ell,0)=[-4\ell,0]$,
    $\gamma(2\ell+1,0)=[-4\ell+2k,1]$,
    $\gamma(2\ell,1)=[4\ell-2k-2,0]$,
    $\gamma(2\ell+1,1)=[4\ell+2,1]$
    for some $k\in \Z$,
    \item
    $\gamma(2\ell,0)=[0,0]$,
    $\gamma(2\ell+1,0)=[2k,1]$,
    $\gamma(2\ell,1)=[2k,1]$,
    $\gamma(2\ell+1,1)=[0,0]$
    for some $k\in \Z$,
    \item
    $\gamma(2\ell,0)=[-4\ell,0]$,
    $\gamma(2\ell+1,0)=[-4\ell+2k,1]$,
    $\gamma(2\ell,1)=[4\ell,1]$,
    $\gamma(2\ell+1,1)=[4\ell-2k,0]$
    for some $k\in \Z$.
  \end{enumerate}
\end{thm}

For each case, we write $\gamma_{(i)}$ for the $\gamma$-map.
We write  $\gamma_{(i,k)}$ ($i\in\{2, 4, 5, 6, 7, 8\}$)
to specify the parameter $k$.

\section{Equivalence of $\gamma$-maps}
We determine the equivalence classes of $\gamma$-maps.

\begin{lemm}\label{lemCindy}
  $\gamma$-maps are characterized by the values at
  $(1,0)$, $(0,1)$, and $(2,0)$.
\end{lemm}

\begin{proof}
  We make a list for the values of $\gamma_{(i)}$. 
  $$
    \begin{array}{c||c|c|c}
      & \gamma_{(i)}(1,0) & \gamma_{(i)}(0,1) &\gamma_{(i)}(2,0)\\
      \hline
      \hline
      (1) & [0,0] & [0,0] & [0,0] \\
      \hline
      (2,k) & [-2,0] & [k,0] & [-4,0]\\
      \hline
      (3) & [-2,0] & [0,1] & [-4,0]\\
      \hline
      (4,k) & [0,0] & [k,1] & [0,0]\\
      \hline
      (5,k) & [2k,1] & [0,0] & [0,0]\\
      \hline
      (6,k) & [2k,1] & [-2k-2,0] & [-4,0]\\
      \hline
      (7,k) & [2k,1] & [2k,1] & [0,0]\\
      \hline
      (8,k) & [2k,1] & [0,1] & [-4,0]\\
      \hline
    \end{array}
  $$
  This shows the lemma.
\end{proof}

\begin{rema}
  Cindy Tsang pointed out that Lemma \ref{lemCindy}
  holds without the classification of $\gamma$-maps.
\end{rema}

\begin{thm}\label{cls_gamma}
  Representatives of $\gamma$-maps on $D_\infty$ are
  $$\gamma_{(1)}, \gamma_{(2,0)}, \gamma_{(2,1)}, \gamma_{(3)},
  \gamma_{(4,0)}, \gamma_{(4,1)}, \gamma_{(5,0)}, \gamma_{(5,1)},
  \gamma_{(6,0)}, \gamma_{(6,1)}. $$
\end{thm}

\begin{proof}
  We will show that
  \begin{itemize}
    \item $\{\gamma_{(1)}\}$ and $\{\gamma_{(3)}\}$
    are equivalence classes,
    \item for $i, j\in\{2,4,5,6\}$, $i\ne j$, $\gamma_{(i,k)}$
    and $\gamma_{(j,k')}$ are not equivalent,
    \item for $i\in \{2,4,5,6, 7, 8\}$ and $k, k'\in \Z$,
    $\gamma_{(i,k)}$ and $\gamma_{(i,k')}$ are equivalent
    if and only if $k\equiv k'\pmod{2}$, 
    \item $\gamma_{(5,0)}$ and $\gamma_{(7,1)}$ are equivalent, 
    \item $\gamma_{(5,1)}$ and $\gamma_{(7,0)}$ are equivalent,
    \item $\gamma_{(6,0)}$ and $\gamma_{(8,1)}$ are equivalent, and
    \item $\gamma_{(6,1)}$ and $\gamma_{(8,0)}$ are equivalent.
  \end{itemize}
  These prove the theorem.
  To show them, we will give a table of
  $\gamma_{(i)}^{[b,t]}(a,s)$ for $(a,s)=(1,0), (0,1), (2,0)$.
  From Lemma \ref{lemCindy},
  it suffices to look at these values to distinguish $\gamma$-maps on $D_\infty$.

  {\small
    $$\begin{array}{c||c|c|c||c}
      & \gamma_{(i)}^{[b,t]}(1,0)
      & \gamma_{(i)}^{[b,t]}(0,1)
      & \gamma_{(i)}^{[b,t]}(2,0) & \gamma_{(j)}\\
      \hline
      \hline
      (1) & [0,0] & [0,0] & [0,0] & \gamma_{(1)}\\
      \hline
      (2,k) & [-2,0] & [(-1)^t(k+2b),0] & [-4,0]& \gamma_{(2,(-1)^t(k+2b))}\\
      \hline
      (3) & [-2,0] & [0,1] & [-4,0] & \gamma_{(3)}\\
      \hline
      (4,k) & [0,0] & [(-1)^t(k-2b),1] & [0,0] & \gamma_{(4,(-1)^t(k-2b))}\\
      \hline
      (5,k), \text{$b$:even} & [(-1)^t2(k-b),1] & [0,0] & [0,0] & \gamma_{(5,(-1)^t(k-b))}\\
      (5,k), \text{$b$:odd} & [(-1)^t2(k-b),1]  & [(-1)^t2(k-b),1] & [0,0]
                                  & \gamma_{(7,(-1)^t(k-b))}\\
      \hline
      (6,k), \text{$b$:even, $t=0$} & [2(k-b),1] & [-2(k-b)-2,0] & [-4,0] &\gamma_{(6,k-b)}\\
      (6,k), \text{$b$:even, $t=1$} & [-2(k-b+2),1] & [2(k-b+2)-2,0] & [-4,0]&\gamma_{(6,-k+b-2)}\\
      (6,k), \text{$b$:odd, $t=0$} & [2(k-b),1] & [0,1] & [-4,0]&\gamma_{(8,k-b)}\\
      (6,k), \text{$b$:odd, $t=1$} & [-2(k-b+2),1] & [0,1] & [-4,0]&\gamma_{(8,-k+b-2)}\\
      \hline
      (7,k), \text{$b$:even}  & [(-1)^t2(k-b),1] & [(-1)^t2(k-b),1] & [0,0]
                                  &\gamma_{(7,(-1)^t(k-b))}\\
      (7,k), \text{$b$:odd}  & [(-1)^t2(k-b),1] & [0,0] & [0,0]
                                  &\gamma_{(5,(-1)^t(k-b))}\\
      \hline
      (8,k), \text{$b$:even, $t=0$}& [2(k-b),1] & [0,1] & [-4,0]&\gamma_{(8,k-b)}\\
      (8,k), \text{$b$:even, $t=1$}& [-2(k-b+2),1] & [0,1] & [-4,0]&\gamma_{(8,-k+b-2)}\\
      (8,k), \text{$b$:odd, $t=0$}& [2(k-b),1] & [-2(k-b)-2,0] & [-4,0]&\gamma_{(6,k-b)}\\
      (8,k), \text{$b$:odd, $t=1$}& [-2(k-b+2),1] & [2(k-b+2)-2,0] & [-4,0]&\gamma_{(6,-k+b-2)}\\
      \hline
    \end{array}$$}
  The three central columns determine what is $\gamma_{(i)}^{[b,t]}$,
  which is shown in the last column.
  By the table, we can check all required conditions.
\end{proof}

Applying $a\ci b=a\gamma(a)(b)$ to Theorem \ref{cls_gamma}, we can show
our main result Theorem \ref{cls_slb}.

\section*{Acknowledgment}
The authors would like to thank the anonymous
referees for their careful reading and many helpful comments.
Also the authors would like to thank Cindy Tsang
and Yuta Kozakai for their valuable comments.
The first author was supported by JSPS KAKENHI Grant Number JP22K03266.

\bibliographystyle{amsplain}
\providecommand{\bysame}{\leavevmode\hbox to3em{\hrulefill}\thinspace}
\providecommand{\MR}{\relax\ifhmode\unskip\space\fi MR }
\providecommand{\MRhref}[2]{%
  \href{http://www.ams.org/mathscinet-getitem?mr=#1}{#2}
}
\providecommand{\href}[2]{#2}

\section*{Appendix A}
This is a supplement for Proposition \ref{prop10} (2).
\medskip

{\tiny
\noindent\underline{$(a,s)=(2\ell,0)$, $(b,t)=(2m,0)$}
\begin{align*}
  \text{LHS} &= p(2\ell,0)+(-1)^{q(2\ell,0)}p(2m,0)=-4\ell-4m,\\
  \text{RHS} &= p(2\ell+2m,0)=-4\ell-4m.
\end{align*}

\noindent\underline{$(a,s)=(2\ell,0)$, $(b,t)=(2m+1,0)$}
\begin{align*}
  \text{LHS} &= p(2\ell,0)+p(2m+1,0)=-4\ell-4m+2k,\\
  \text{RHS} &= p(2\ell+(2m+1),0)=-4(\ell+m)+2k.
\end{align*}

\noindent\underline{$(a,s)=(2\ell,0)$, $(b,t)=(2m,1)$}
\begin{align*}
  \text{LHS} &= p(2\ell,0)+p(2m,1)=-4\ell+4m-2k-2, \\
  \text{RHS} &= p(2\ell+2m+p(2\ell,0),1)=p(-2\ell+2m,1)= -4\ell+4m-2k-2.
\end{align*}

\noindent\underline{$(a,s)=(2\ell,0)$, $(b,t)=(2m+1,1)$}
\begin{align*}
  \text{LHS} &= p(2\ell,0)+p(2m+1,1)=-4\ell+4m+2,\\
  \text{RHS} &= p(2\ell+(2m+1)-4\ell,1)=p(-2\ell+2m+1,1)= -4\ell+4m+2.
\end{align*}

\noindent\underline{$(a,s)=(2\ell+1,0)$, $(b,t)=(2m,0)$}
\begin{align*}
  \text{LHS} &= p(2\ell+1,0)-p(2m,0)=-4\ell+2k+4m,\\
  \text{RHS} &= p(2\ell+1-2m,0)=-4(\ell-m)+2k.
\end{align*}

\noindent\underline{$(a,s)=(2\ell+1,0)$, $(b,t)=(2m+1,0)$}
\begin{align*}
  \text{LHS} &= p(2\ell+1,0)-p(2m+1,0)=-4\ell+2k-(-4m+2k)=-4\ell+4m,\\
  \text{RHS} &= p(2\ell+1-(2m+1),0)=-4\ell+4m.
\end{align*}

\noindent\underline{$(a,s)=(2\ell+1,0)$, $(b,t)=(2m,1)$}
\begin{align*}
  \text{LHS} &= p(2\ell+1,0)-p(2m,1)=-4\ell+2k-(4m-2k-2)=-4\ell-4m+4k+2,\\
  \text{RHS} &= p(2\ell+1-2m+p(2\ell+1,0),1)= p(2\ell-2m+1-4\ell+2k,1)
                                   =-4\ell-4m+4k+2.
\end{align*}

\noindent\underline{$(a,s)=(2\ell+1,0)$, $(b,t)=(2m+1,1)$}
\begin{align*}
  \text{LHS} &= p(2\ell+1,0)-p(2m+1,1)=-4\ell+2k-(4m+2)=-4\ell-4m+2k-2,\\
  \text{RHS} &= p(2\ell+1-(2m+1)+(-4\ell+2k),1)= p(-2\ell-2m+2k,1)
                               =-4\ell-4m+2k-2.
\end{align*}

\noindent\underline{$(a,s)=(2\ell,1)$, $(b,t)=(2m,0)$}
\begin{align*}
  \text{LHS} &= p(2\ell,1)+p(2m,0)=4\ell-2k-2-4m,\\
  \text{RHS} &= p(2\ell-2m,1)=4\ell-4m-2k-2.
\end{align*}

\noindent\underline{$(a,s)=(2\ell,1)$, $(b,t)=(2m+1,0)$}
\begin{align*}
  \text{LHS} &= p(2\ell,1)+p(2m+1,0)=4\ell-4m-2,\\
  \text{RHS} &= p(2\ell-(2m+1),1)=4\ell-4m-2.
\end{align*}

\noindent\underline{$(a,s)=(2\ell,1)$, $(b,t)=(2m,1)$}
\begin{align*}
  \text{LHS} &= p(2\ell,1)+p(2m,1)=4\ell-2k-2+4m-2k-2=4\ell+4m-4k-4,\\
  \text{RHS} &= p(2\ell-2m-(4\ell-2k-2),0)=p(-2\ell-2m+2k+2,0)
                                   =4\ell+4m-4k-4.
\end{align*}

\noindent\underline{$(a,s)=(2\ell,1)$, $(b,t)=(2m+1,1)$}
\begin{align*}
  \text{LHS} &= p(2\ell,1)+p(2m+1,1)=4\ell-2k-2+4m+2=4\ell+4m-2k,\\
  \text{RHS} &= p(2\ell-(2m+1)-(4\ell-2k-2),0)
  = p(-2\ell-2m+2k+1,0)=4\ell+4m-4k+2k=4\ell+4m-2k.
\end{align*}

\noindent\underline{$(a,s)=(2\ell+1,1)$, $(b,t)=(2m,0)$}
\begin{align*}
  \text{LHS} &= p(2\ell+1,1)-p(2m,0)=4\ell+2+4m,\\
  \text{RHS} &= p(2\ell+1+2m,1)=4\ell+4m+2.
\end{align*}

\noindent\underline{$(a,s)=(2\ell+1,1)$, $(b,t)=(2m+1,0)$}
\begin{align*}
  \text{LHS} &= p(2\ell+1,1)-p(2m+1,0)
                                   =4\ell+2-(-4m+2k)=4\ell+4m-2k+2,\\
  \text{RHS} &= p(2\ell+1+(2m+1),1)
  =4\ell+4m-2k+2.
\end{align*}

\noindent\underline{$(a,s)=(2\ell+1,1)$, $(b,t)=(2m,1)$}
\begin{align*}
  \text{LHS} &= p(2\ell+1,1)-p(2m,1)=4\ell+2-(4m-2k-2)=4\ell-4m+2k+4,\\
  \text{RHS} &= p(2\ell+1+2m-(4\ell+2),0)
  = p(-2\ell+2m-1,0)=4\ell-4m+4+2k.
\end{align*}

\noindent\underline{$(a,s)=(2\ell+1,1)$, $(b,t)=(2m+1,1)$}
\begin{align*}
  \text{LHS} &= p(2\ell+1,1)-p(2m+1,1)=4\ell+2-(4m+2)=4\ell-4m,\\
  \text{RHS} &= p(2\ell+1+2m+1-(4\ell+2),0)
  = p(-2\ell+2m,0)=4\ell-4m.
\end{align*}
}

\section*{Appendix B}
This is a supplement for Proposition \ref{prop11} (2).
\medskip

{\tiny
\noindent\underline{$(a,s)=(2\ell,0)$, $(b,t)=(2m,0)$}
\begin{align*}
  \text{LHS} &= p(2\ell,0)+p(2m,0)=-4\ell-4m,\\
  \text{RHS} &= p(2\ell+2m,0)=-4\ell-4m.
\end{align*}

\noindent\underline{$(a,s)=(2\ell,0)$, $(b,t)=(2m+1,0)$}
\begin{align*}
  \text{LHS} &= p(2\ell,0)+p(2m+1,0)=-4\ell-4m+2k,\\
  \text{RHS} &= p(2\ell+2m+1,0)=-4\ell-4m+2k.
\end{align*}

\noindent\underline{$(a,s)=(2\ell,0)$, $(b,t)=(2m,1)$}
\begin{align*}
  \text{LHS} &= p(2\ell,0)+p(2m,1)=-4\ell+4m,\\
  \text{RHS} &= p(2\ell+2m+p(2\ell,0),1)=p(2\ell+2m-4\ell,1)=-4\ell+4m.
\end{align*}

\noindent\underline{$(a,s)=(2\ell,0)$, $(b,t)=(2m+1,1)$}
\begin{align*}
  \text{LHS} &= p(2\ell,0)+p(2m+1,1)=-4\ell+4m-2k,\\
  \text{RHS} &= p(-2\ell+2m+1,1)=-4\ell+4m-2k.
\end{align*}

\noindent\underline{$(a,s)=(2\ell+1,0)$, $(b,t)=(2m,0)$}
\begin{align*}
  \text{LHS} &= p(2\ell+1,0)-p(2m,0)=-4\ell+2k+4m,\\
  \text{RHS} &= p(2\ell+1-2m,0)=-4\ell+4m+2k.
\end{align*}

\noindent\underline{$(a,s)=(2\ell+1,0)$, $(b,t)=(2m+1,0)$}
\begin{align*}
  \text{LHS} &= p(2\ell+1,0)-p(2m+1,0)=-4\ell+2k+4m-2k=-4\ell+4m,\\
  \text{RHS} &= p(2\ell+1-(2m+1),0)=-4\ell+4m.
\end{align*}

\noindent\underline{$(a,s)=(2\ell+1,0)$, $(b,t)=(2m,1)$}
\begin{align*}
  \text{LHS} &= p(2\ell+1,0)-p(2m,1)=-4\ell-4m+2k,\\
  \text{RHS} &= p(2\ell+1-2m+p(2\ell+1,0),1)
  =p(-2\ell-2m+1+2k,1)=-4\ell-4m+2k.\\
\end{align*}

\noindent\underline{$(a,s)=(2\ell+1,0)$, $(b,t)=(2m+1,1)$}
\begin{align*}
  \text{LHS} &= p(2\ell+1,0)-p(2m+1,1)=-4\ell+2k-4m+2k=-4\ell-4m+4k,\\
  \text{RHS} &= p(2\ell+1-(2m+1)+p(2\ell+1,0),1)=-4\ell-4m+4k.
\end{align*}

\noindent\underline{$(a,s)=(2\ell,1)$, $(b,t)=(2m,0)$}
\begin{align*}
  \text{LHS} &= p(2\ell,1)-p(2m,0)=4\ell+4m,\\
  \text{RHS} &= p(2\ell+2m,1)=4\ell+4m.
\end{align*}

\noindent\underline{$(a,s)=(2\ell,1)$, $(b,t)=(2m+1,0)$}
\begin{align*}
  \text{LHS} &= p(2\ell,1)-p(2m+1,0)=4\ell-(-4m+p(1,0))=4\ell+4m-2k,\\
  \text{RHS} &= p(2\ell+2m+1,1)=4\ell+4m-2k. 
\end{align*}

\noindent\underline{$(a,s)=(2\ell,1)$, $(b,t)=(2m,1)$}
\begin{align*}
  \text{LHS} &= p(2\ell,1)-p(2m,1)=4\ell-4m,\\
  \text{RHS} &= p(2\ell+2m-p(2\ell,1),0)=p(-2\ell+2m,0)=4\ell-4m.
\end{align*}

\noindent\underline{$(a,s)=(2\ell,1)$, $(b,t)=(2m+1,1)$}
\begin{align*}
  \text{LHS} &= p(2\ell,1)-p(2m+1,1)=4\ell-4m+2k, \\
  \text{RHS} &= p(2\ell+(2m+1)-p(2\ell,1),0)
                                   =p(-2\ell+2m+1,0)=4\ell-4m+2k.
\end{align*}

\noindent\underline{$(a,s)=(2\ell+1,1)$, $(b,t)=(2m,0)$}
\begin{align*}
  \text{LHS} &= p(2\ell+1,1)+p(2m,0)=4\ell-2k-4m,\\
  \text{RHS} &= p(2\ell+1-2m,1)=4\ell-4m-2k.
\end{align*}

\noindent\underline{$(a,s)=(2\ell+1,1)$, $(b,t)=(2m+1,0)$}
\begin{align*}
  \text{LHS} &= p(2\ell+1,1)+p(2m+1,0)=4\ell-2k-4m+2k=4\ell-4m,\\
  \text{RHS} &= p(2\ell+1-(2m+1),1)=p(2\ell-2m,1)=4\ell-4m.
\end{align*}

\noindent\underline{$(a,s)=(2\ell+1,1)$, $(b,t)=(2m,1)$}
\begin{align*}
  \text{LHS} &= p(2\ell+1,1)+p(2m,1)=4\ell+4m-2k,\\
  \text{RHS} &= p(2\ell+1-2m-p(2\ell+1,1),0)
                                   =p(-2\ell-2m+1+2k,0)=4\ell+4m-2k.
\end{align*}

\noindent\underline{$(a,s)=(2\ell+1,1)$, $(b,t)=(2m+1,1)$}
\begin{align*}
  \text{LHS} &=  p(2\ell+1,1)+p(2m+1,1)=4\ell-2k+4m-2k=4\ell+4m-4k,\\
  \text{RHS} &= p(2\ell+1-(2m+1)-(4\ell-2k),0)
                                   =p(-2\ell-2m+2k,0)=4\ell+4m-4k.
\end{align*}
}

\end{document}